\documentclass[a4paper,leqno]{amsproc}

\usepackage{amsmath,amssymb,amsthm}
\usepackage[colorlinks=true]{hyperref}

\def\dis
{\displaystyle}

\def\R{{\mathbf R}}
\def\N{{\mathbf N}}

\def\virgp{\raise 2pt\hbox{,}}
\def\({\left(}
\def\){\right)}
\def\<{\left\langle}
\def\>{\right\rangle}
\def\le{\leqslant}
\def\ge{\geqslant}

\def\Eq#1#2{\mathop{\sim}\limits_{#1\rightarrow#2}}
\def\Tend#1#2{\mathop{\longrightarrow}\limits_{#1\rightarrow#2}}

\def\d{{\partial}}

\def\Norm#1#2{\left\lVert #1 \right\rVert_{#2}}

\def\conjug{\mathcal C}

\theoremstyle{plain}
\newtheorem{theorem}{Theorem}[section]
\newtheorem{lemma}[theorem]{Lemma}

\newtheorem{proposition}[theorem]{Proposition}
\newtheorem*{concl}{Conclusion}

\theoremstyle{definition}
\newtheorem{definition}[theorem]{Definition}

\newtheorem{remark}[theorem]{Remark}

\numberwithin{equation}{section}

\begin{document}

\title[Rotating points for NLS scattering]{Rotating
  points for the conformal NLS scattering operator}    
\author[R. Carles]{R\'emi Carles}
\address{Universit\'e Montpellier~2\\ 
Math\'ematiques,
  CC051\\34095 
  Montpellier\\ France}
\address{CNRS\\ UMR 5149\\34095 
  Montpellier\\ France}
\email{Remi.Carles@math.cnrs.fr}
\begin{abstract}
We consider the  nonlinear
Schr\"odinger equation, with mass-critical nonlinearity, focusing or
defocusing. For any given angle, we establish 
the existence of infinitely many functions on which the 
scattering operator acts as a rotation of this angle. Using a lens
transform, we reduce the problem to the existence of a solution to a
nonlinear Schr\"odinger equation with harmonic potential, satisfying
suitable periodicity 
properties. The existence of infinitely many such solutions is
proved thanks to a constrained minimization problem.  
\end{abstract}
\thanks{This work was supported by the French ANR project
  R.A.S. (ANR-08-JCJC-0124-01).} 
\maketitle

\section{Introduction}
\label{sec:intro}
We consider the pseudo-conformally invariant nonlinear Schr\"odinger
equation 
\begin{equation}
  \label{eq:nls}
  i\d_t u + \frac{1}{2}\Delta u = \lvert u\rvert^{4/d} u,\quad
  (t,x)\in \R\times \R^d,\ d\ge 1.
\end{equation}
Two types of initial data are of special interest:
\begin{align}
  &\text{Asymptotic state: }U_0(-t)u(t)\big|_{t=\pm
  \infty}=u_\pm,\quad \text{where
  }U_0(t)=e^{i\frac{t}{2}\Delta}.\label{eq:CIasym} \\
&\text{Cauchy data at }t=0:\ u_{\mid t=0} = u_0.\label{eq:CIcauchy}
\end{align}
It is well known that for data $u_\pm\in \Sigma=H^1\cap{\mathcal F}(H^1)$, 
where
\begin{equation}\label{eq:fourier}
  {\mathcal F} f(\xi)=\widehat
  f(\xi)=\frac{1}{(2\pi)^{d/2}}\int_{\R^d}f(x)e^{-ix\cdot \xi} dx,
\end{equation}
\eqref{eq:nls}--\eqref{eq:CIasym} has a unique, global, solution
$u\in C(\R;\Sigma)$ (\cite{GV79Scatt}, see also
\cite{CazCourant}). Its initial value $u_{\mid t=0}$ is the image of 
the asymptotic state under the action of the wave operator:
\begin{equation*}
  u_{\mid t=0}= W_\pm u_\pm. 
\end{equation*}
Similarly, if $u_0\in \Sigma$, \eqref{eq:nls}--\eqref{eq:CIcauchy}
possesses asymptotic states:
\begin{equation*}
  \exists u_\pm\in \Sigma,\quad
  \|U_0(-t)u(t)-u_\pm\|_\Sigma \Tend t {\pm \infty} 0: \quad 
  u_\pm=W_\pm^{-1}u_0.
\end{equation*}
The scattering operator associated to \eqref{eq:nls} is classically
defined as
\begin{equation*}
  S=W_+^{-1}\circ W_-: u_-\mapsto u_+.
\end{equation*}
It maps $\Sigma$ to $\Sigma$, and is unitary on $L^2$, and on $\dot H^1$:
\begin{equation*}
  \lVert S(u_-)\rVert_{L^2(\R^d)}=\lVert u_-\rVert_{L^2(\R^d)}\quad
  ;\quad \lVert \nabla S(u_-)\rVert_{L^2(\R^d)}=\lVert \nabla
  u_-\rVert_{L^2(\R^d)}. 
\end{equation*}
This follows from \cite{GV79Scatt,TsutsumiSigma,HT87}. 
\smallbreak

Besides the existence of the wave and scattering operators, it seems
that very few of their properties are known. By construction, these
operators are continuous on $\Sigma$. When $d\le 2$ (the
nonlinearity is smooth), these operators are real analytic on
$\Sigma$; see \cite{CG09}.   
\smallbreak

It is rather reassuring to check that the operators $W_\pm$ and $S$
are not trivial, showing that averaged in time nonlinear
effects may not be negligible. Following \cite{PG96} for the case of
the wave equation, we can prove for instance that in $L^2(\R^d)$, and
as $\varepsilon\to 0$, 
\begin{equation}\label{eq:DA0}
 S \left(\varepsilon u_-\right) = \varepsilon u_- -
 i\varepsilon^{1+4/d}P(u_-) + {\mathcal O}\left(\varepsilon^{1+8/d}\right), 
\end{equation}
where 
\begin{equation*}
 P(u_-)=  \int_{-\infty}^{+
 \infty} U_0(-t)\left( |U_0(t)u_-|^{4/d}U_0(t)u_-\right)dt.
\end{equation*}
We refer to \cite{COMRL} for a proof (in the present small data case,
it suffices to assume that $u_-\in L^2(\R^d)$). Explicit computations
show that $P(u_-)\not =0$ when $u_-$ is Gaussian, therefore $S$ is
not the identity. 
\smallbreak

A few algebraic properties are available. Let $\conjug$ denote the
conjugation $f\mapsto \overline f$. The invariances of the equation show that
\begin{equation}\label{eq:conj}
  W_\pm = \conjug \circ W_\mp\circ\conjug ,
\end{equation}
an identity which was noticed in \cite{CW92} (see also
\cite{CazCourant}).  Due to the invariance of \eqref{eq:nls} under
translation and gauge transforms, 
\begin{equation}\label{eq:invariant}
  \begin{aligned}
    &S\( u_-(\cdot +a)\)= S\(u_-\)(\cdot +a),\ \forall a\in\R^d,\\
  &S\( e^{i\eta}u_-\)= e^{i\eta}S(u_-),\ \forall
  \eta\in\R.
\end{aligned}
\end{equation}
Another 
algebraic relation was established in \cite{COMRL}, which seems to be bound to
the conformal case, contrary to \eqref{eq:conj} and \eqref{eq:invariant}:
  \begin{equation*}
    {\mathcal F}\circ W_\pm^{-1}=W_\mp\circ {\mathcal F}.
  \end{equation*}
In \cite{BG99}, a remarkable property was proved for the scattering
operator associated to the energy-critical wave equation
\begin{equation}\label{eq:nlw}
  \d_t^2 u-\Delta u+ \lvert u\rvert^4 u=0\quad ;\quad x\in \R^3. 
\end{equation}
Using the notion of profile decomposition, as introduced in
\cite{PG98,MetivierSchochetDMJ}, the authors prove in \cite{BG99} that
the scattering operator associated to \eqref{eq:nlw}, and defined on
the energy space, enjoys a surprising nonlinear superposition
principle. It follows from \cite{Keraani01,CKSTTAnnals} that a
similar result holds for the Schr\"odinger analogue of 
\eqref{eq:nlw},
\begin{equation*}
  i\d_t u +\frac{1}{2}\Delta u = \lvert u\rvert^{4}u\quad ;\quad x\in
  \R^3.
\end{equation*}
We give more precise statements in an appendix, where we also discuss
the case of \eqref{eq:nls}. 
In this paper, we show the existence of infinitely many
fixed points for $S$, and more generally:
\begin{theorem}\label{theo:main}
  Let $d\ge 1$ and $S:\Sigma\to \Sigma$ be the scattering operator
  associated to \eqref{eq:nls}. For any $\theta\in [0,2\pi[$, there are
  infinitely many functions $u_-\in \Sigma$ such that
  $S(u_-)=e^{i\theta}u_-$. 
\end{theorem}
\begin{remark}
  For $\theta\in ]0,2\pi[$, this result yields another evidence that
  $S$ is not trivial.  
\end{remark}
\begin{remark}
  For $\theta=0$, this shows the existence of infinitely many fixed
  points. Similarly, for $\theta = 2\pi p/q$, $p,q\in \N^*$, this shows
  the existence of infinitely many periodic (or cyclic) points, for any given
  period. 
\end{remark}
\begin{remark}\label{rem:forme}
  We construct solutions of the form
  \begin{equation*}
    u(t,x) = \frac{1}{\(1+t^2\)^{d/4}}
    e^{i\frac{t}{1+t^2}\frac{|x|^2}{2}
    -i\(\frac{d}{2}+2j-\frac{\theta}{\pi}\)\operatorname{arctan} t
    }\phi_j\(\frac{x}{\sqrt{1+t^2}}\), 
  \end{equation*}
where $j\in \N\setminus\{0\}$ and $\phi_j\in \Sigma$. The profile
$\phi_j$ is given by the nonlinear eigenvalue equation \eqref{eq:stat}
with $\nu=d/2+2j+\theta/\pi$. 
\end{remark}
\begin{remark}
In semi-classical analysis, the scattering operator appears in some
  cases to describe solutions which pass through a focal point. In the
  presence of an isotropic (but not necessarily) harmonic potential,
  focusing at one point occurs periodically in time, and the
  scattering operator is iterated each time a focal point is
  traversed; see \cite{CaBook} and references therein. The existence of
  fixed points, and more 
  generally, of periodic points, shows that the nonlinear dynamics may
  reveal some periodicity in time, at leading order in the
  semi-classical limit.  
\end{remark}
\begin{remark}
  Not all the functions in $\Sigma$ are such that
  $S(u_-)(x)=e^{ih(x)}u_-(x)$ for some real-valued function $h$ (not
  necessarily constant). Arguing by contradiction and using
  \eqref{eq:DA0}, one can show that there exist two functions $u_-$
  and $\widetilde u_-$ in $\Sigma$ such that
  \begin{equation*}
    \lvert u_- (x)\rvert \equiv \lvert \widetilde u_- (x)\rvert, \quad
    \text{and} \quad
\lvert S(u_-)(x) \rvert \not\equiv \lvert S(\widetilde u_-)(x) \rvert.
  \end{equation*}
See \cite[\S7.4.3]{CaBook}. 
\end{remark}
\begin{remark}
  For the linear Schr\"odinger equation, one can construct
  \emph{transparent} potentials $V(x)$. This means that one can choose
  a potential $V$ such that any function $u(t,x) =e^{iEt}\psi(x)$,
  $E\in \R$,  solution to 
  \begin{equation*}
    i\d_t u+\frac{1}{2}\Delta u = Vu
  \end{equation*}
has a trivial scattering matrix; see \cite{Kerimov07} and references
therein. 
\end{remark}
In the focusing case
\begin{equation}
  \label{eq:nlsfoc}
  i\d_t u + \frac{1}{2}\Delta u = -\lvert u\rvert^{4/d} u,\quad
  (t,x)\in \R\times \R^d,
\end{equation}
no general scattering theory is available, since finite time blow-up
may occur (see e.g. \cite{CazCourant,TaoDisp}). We know however that
for (initial or asymptotic) data with a sufficiently small $L^2$ norm,
the solution is global, and there is scattering \cite{CW89}. Recall
that the ground state given as the unique positive, radially symmetric
solution to 
\begin{equation}
  \label{eq:Q}
  -\frac{1}{2}\Delta Q+Q=Q^{1+4/d},
\end{equation}
yields the best constant for the following Gagliardo--Nirenberg
inequality \cite{Weinstein83}:
 \begin{equation}\label{eq:GNopt}
   \|f\|_{L^{2+4/d}}^{2+4/d} \le
   \frac{d+2}{2d\|Q\|_{L^2}^{4/d}}\|f\|_{L^2}^{4/d} \|\nabla
   f\|_{L^2}^2,\quad \forall f\in H^1\big(\R^d\big). 
 \end{equation}
If $\|u\|_{L^2}<\|Q\|_{L^2}$, then all
$H^1$-solutions to \eqref{eq:nlsfoc} are global in time
\cite{Weinstein83}. It is conjectured that
the same holds true for $L^2$-solutions; see
\cite{MerleTsutsumi,KeraaniJFA,HK06,Tzirakis06,VisanZhang07,KTV,TVZradial,KVZ07}, 
and references therein.  
\begin{theorem}\label{theo:foc}
  Let $d\ge 1$. For any $\theta\in [0,2\pi[$, there are
  infinitely many functions $u_-\in \Sigma$ such that the scattering
  operator associated to the focusing equation \eqref{eq:nlsfoc} is
  well defined on $u_-$, and such that
  $S(u_-)=e^{i\theta}u_-$. In addition,
  \begin{itemize}
  \item These functions satisfy $\displaystyle
    \|u_-\|_{L^2}>\(\frac{d}{d+2}\)^{d/4} \|Q\|_{L^2}$.
\item They are arbitrarily large in $H^1(\R^d)$ (resp. in $\dot
  H^1(\R^d)$ if $d\ge 2$). 
\item For all such $u_-\in \Sigma$, there exists $\varepsilon>0$ such that if
 $\widetilde u_-\in \Sigma$ satisfies $\|u_--\widetilde
 u_-\|_{L^2}<\varepsilon$, then $S(\widetilde u_-)$ is 
  well-defined in $\Sigma$ (in particular, there is no blow-up). 
  \end{itemize}
\end{theorem}
\begin{remark}
  We construct solutions of the same form as in
  Remark~\ref{rem:forme}, but with different profiles $\phi$, and in
  the phase, $-j\in \N$, with $-2j>d/2$. 
\end{remark} 
\begin{remark}
  In the defocusing case, the last property stated above is
  straightforward, since $S$ is defined on $\Sigma$. In the above
  focusing case, this stability property is more surprising. 
\end{remark}
Theorems~\ref{theo:main} and \ref{theo:foc} rely on two
steps. As shown in \S\ref{sec:reduc}, a lens transform reduces the
proof to the existence of time periodic solutions for the equation
\begin{equation*}
  i\d_t v +\frac{1}{2}\Delta v = \frac{\lvert x\rvert^2}{2} v \pm
  \lvert v\rvert^{4/d}v. 
\end{equation*}
Since the nonlinearity is autonomous, it is reasonable to expect
solutions to the above equation which are standing waves,
$v(t,x)=e^{-i\nu t}\psi(x)$. The point is that infinitely many values
for $\nu$ lead to periodic solutions with a \emph{suitable} period, that is,
such that by inverting the lens transform, we get 
Theorem~\ref{theo:main} and the existence part of
Theorem~\ref{theo:foc}. This step is achieved in 
\S\ref{sec:periodic}. The rest of the proof of Theorem~\ref{theo:foc}
is given in \S\ref{sec:focusing}.  Finally, in the appendix, we
discuss the nonlinear 
superposition principle associated to the scattering operator for
\eqref{eq:nls}, modulo some global existence issues which are still
open so far.

\section{Reduction of the problem}
\label{sec:reduc}

\subsection{Lens transform}
\label{sec:lens}

Let $u\in C(\R;\Sigma)$ solve the more general equation
\begin{equation}
  \label{eq:nlsgen}
  i\d_t u + \frac{1}{2}\Delta u = \lvert u\rvert^{2\sigma} u,\quad
  (t,x)\in \R\times \R^d.
\end{equation}
If 
\begin{equation}\label{eq:puissance}
 \begin{aligned}
  \sigma_0(d)& < \sigma<\frac{2}{d-2} \quad\(\text{with only
  }\sigma>\sigma_0(d)\text{ if }d\le 2\), \\
\text{ where
  }\sigma_0(d)&:=\frac{2-d+\sqrt{d^2+12d+4}}{4d} ,
\end{aligned} 
\end{equation}
then the scattering operator associated to \eqref{eq:nlsgen} is well
defined, from $\Sigma$ to $\Sigma$;
\cite{GV79Scatt,TsutsumiSigma,HT87} (see also 
\cite{CW92,NakanishiOzawa} where the case $\sigma=\sigma_0(d)$ is allowed).
Introduce
\begin{equation}
  \label{eq:lens}
  v(t,x)=\frac{1}{\(\cos t\)^{d/2}} u\(\tan t,\frac{x}{\cos t}\) e^{-i
  \frac{|x|^2}{2}\tan t},
\end{equation}
which is well defined for $|t|<\pi/2$, and has the same value as $u$
at time $t=0$. As noticed in
\cite{KavianWeissler,Rybin}  (see also \cite{CaM3AS,TaoLens}), $v$
solves, at least formally:
\begin{equation}
  \label{eq:nlsharmo}
  i\d_t v +\frac{1}{2}\Delta v = \frac{\lvert x\rvert^2}{2} v + \lvert \cos
  t\rvert^{d\sigma-2} \lvert v\rvert^{2\sigma}v.
\end{equation}
Note an important feature of the lens transform \eqref{eq:lens}: it
maps the line $\R_t$ for $u$, to the bounded interval
$]-\frac{\pi}{2},\frac{\pi}{2}[$ for $v$. Therefore, long time
properties for $u$ are equivalent to local in time properties for
$v$.

\subsection{The harmonic oscillator}
\label{sec:harmo}

Let 
\begin{equation*}
  H=-\frac{1}{2}\Delta +\frac{\lvert x\rvert^2}{2}, \text{ and }U_H(t)=e^{-itH}
\end{equation*}
denote the harmonic oscillator and its propagator. Recall some
well-known properties (see e.g. \cite{LandauQ}):
\begin{lemma}\label{lem:harmo}
  We have
  \begin{equation*}
    \sigma_{p}(H)=\left\{ \frac{d}{2}+k=:\l_k \ ;\ k\in \N\right\},
  \end{equation*}
and the associated eigenfunctions are given by (tensor products of)
Hermite functions, which form a basis of $L^2(\R^d)$. For $\ell\in
\N$, the
eigenfunctions associated to $\l_{2\ell}$ (resp. $\l_{2\ell+1}$) are
even (resp. odd). 
\end{lemma}
In space dimension $d=1$, these eigenvalues are simple, and the
eigenvectors are given by Hermite functions  
$(\psi_k)_{k\in \N}$. For even indices, $\psi_{2\ell}$ is
even, and for odd indices, $\psi_{2\ell+1}$ is odd. Up to normalizing
constants, we have for instance
\begin{equation*}
  \psi_0(x)= e^{-x^2/2}\quad ;\quad \psi_1(x)= xe^{-x^2/2}.
\end{equation*}
In higher dimensions, one considers tensor products of one-dimensional
eigenfunctions; the eigenvalues $\l_k>d/2$ are no longer simple. We
note the identity
\begin{equation}\label{eq:maslov}
  U_H\(t+\pi\)\psi(x)=e^{-id\pi/2} U_H(t)\psi(-x),\quad \forall t\in \R.
\end{equation}

Finally, as a direct consequence of Mehler's formula, local in time
Strichartz estimates are available for $U_H$ (see
e.g. \cite{CaAHP,CazCourant}). Notice that since $H$ possesses
eigenvalues, global in time Strichartz estimates fails for $U_H$. 
 \begin{lemma}\label{lem:strichartz}
  Let $d\ge 1$. A pair $(p,q)$ is admissible provided that 
  \begin{equation*}
\frac{2}{p}+\frac{d}{q}=\frac{d}{2},\quad p\ge 2, \quad
(p,q,d)\neq(2,\infty,2). 
\end{equation*}
Consider a \emph{finite} time interval $I$. \\
$1.$ For all admissible pair $(p,q)$, there exists $C=C(p,|I|)$ such
that
\begin{equation*}
  \left\lVert U_H(\cdot)\phi\right\rVert_{L^p(I;L^q(\R^d))}\le C
  \|\phi\|_{L^2},\quad \forall \phi\in L^2(\R^d). 
\end{equation*}
$2.$ Define the retarded operator in defined by
\begin{equation*}
  R(F)(t,x) = \int_{I\cap \{s\le t\}} U_H(t-s)F(s,x)ds. 
\end{equation*}
For all admissible pairs $(p_1,q_1)$, $(p_2,q_2)$, there
exists $C=C(p_1,p_2,|I|)$ with
\begin{equation*}
  \left\lVert R(F)\right\rVert_{L^{p_1}(I;L^{q_1}(\R^d))}\le C
  \|F\|_{L^{p_2'}(I;L^{q_2'}(\R^d))},\quad \forall F\in L^{p_2'}(I;L^{q_2'}(\R^d)).
\end{equation*}
\end{lemma}
\subsection{A rotating point for $S$}
\label{sec:onefixed}

The following lemma is standard (see
\cite{TsutsumiSigma} or  \cite{Rauch91}):
\begin{lemma}\label{lem:tsutsumi}
  Let $f\in L^2(\R^d)$, and recall that
  $U_0(t)=e^{i\frac{t}{2}\Delta}$. 
  \begin{equation*}
    \left\lVert U_0(t) f - A(t)f\right\rVert_{L^2(\R^d)} \Tend t
    {\pm\infty} 0, \quad \text{where }A(t)f(x) = \frac{1}{(it)^{d/2}}\widehat
  f\(\frac{x}{t}\)e^{i\frac{\lvert x\rvert^2}{2t}},
  \end{equation*}
and the Fourier transform is normalized in \eqref{eq:fourier}. 
\end{lemma}
\begin{proof}
  From the explicit formula
  \begin{equation*}
    U_0(t)f(x) = \frac{1}{(2i\pi t)^{d/2}}\int_{\R^d}e^{i\frac{\lvert
    x-y\rvert^2}{2t}} f(y)dy,
  \end{equation*}
we have the following factorization:
\begin{equation*}
  U_0(t) = M_t D_t {\mathcal F} M_t,
\end{equation*}
where $M_t$ stands for the multiplication by the function $e^{i\frac{\lvert
    x\rvert^2}{2t}}$, ${\mathcal F}$ is the Fourier transform defined in
    \eqref{eq:fourier}, and $D_t$ is the dilation operator
    \begin{equation*}
      \(D_t f\)(x) = \frac{1}{(it)^{d/2}}f\(\frac{x}{t}\). 
    \end{equation*}
The lemma thus reads: $U_0(t)- M_t D_t {\mathcal F}\to 0$ strongly in
$L^2$, as $t\to
\pm \infty$. Since $M_t D_t {\mathcal F}$ is unitary on $L^2$, the lemma follows
from the strong limit in $L^2$, $M_t-{\rm Id}\to 0$, 
which stems from the Dominated Convergence Theorem.  
\end{proof}
\begin{lemma}\label{lem:waveharmo}
  Let $\sigma$ satisfy \eqref{eq:puissance}, and $u,v \in C(\R;\Sigma)$ solve
  \eqref{eq:nlsgen}  and 
  \eqref{eq:nlsharmo}, respectively, with $u_{\mid t=0}=v_{\mid
  t=0}=\phi$. Then
\begin{equation*}
 v\(-\frac{\pi}{2},x\)=e^{id\pi/4}{\mathcal F}\(
  W_-^{-1}\phi\)(-x) \quad ;\quad v\(\frac{\pi}{2},x\)= e^{-id\pi/4} {\mathcal F}\(
  W_+^{-1}\phi\)(x) .
\end{equation*}
\end{lemma}
\begin{remark}
  In the linear case, the same result holds when $W_\pm$ are replaced
  by $\rm Id$: see \eqref{eq:maslov}. The $-d\pi/2$ phase shift
  between the two instants $\pm\pi/2$ corresponds to the Maslov index,
  and the symmetry with respect to the origin accounts for the fact
  that the harmonic oscillator rotates the phase space with angular
  velocity equal to one.  
\end{remark}
\begin{remark}
  It would suffice to consider $v\in
  C([-\pi/2,\pi/2];L^2(\R^d))$, and $u\in C(\R;L^2(\R^d))$ which has
  asymptotic states in $L^2(\R^d)$. This is guaranteed if we assume
  further spatial regularity as above. In the case $\sigma=2/d$, we might
  also consider either data with small $L^2$ norm \cite{CW89}, or radially
  symmetric $L^2$ functions when $d\ge 2$
  \cite{KTV,TVZradial}. However, for the construction of periodic 
  solutions to \eqref{eq:nlsharmo}, we take advantage of
  properties such as the compactness of the embedding
  $\Sigma\hookrightarrow L^2\cap L^{2\sigma+2}$. 
\end{remark}
\begin{proof}
  We show:
\begin{equation*}
\left\lVert v(t)- e^{-id\pi/4} {\mathcal F}\(
  W_+^{-1}\phi\)\right\rVert_{L^2(\R^d)}\to 0 \text{ as
  }t\mathop{\longrightarrow}\limits_{<}  \frac{\pi}{2}.
\end{equation*}
By  \eqref{eq:lens} and asymptotic completeness for \eqref{eq:nls}, we
  have, in $L^2(\R^d)$:
  \begin{align*}
    v(t,x) =&\frac{1}{\(\cos t\)^{d/2}}e^{-i
  \frac{|x|^2}{2}\tan t} u\(\tan t,\frac{x}{\cos t}\) \\
\Eq t {\frac{\pi}{2}^-}&\frac{1}{\(\cos t\)^{d/2}}  e^{-i
  \frac{|x|^2}{2}\tan t} \(U_0(\tan t) u_+\)\(\frac{x}{\cos t}\),
  \end{align*}
where $u_+=W_+^{-1}\phi$. By Lemma~\ref{lem:tsutsumi}, we infer
\begin{equation*}
  v(t,x)\Eq t {\frac{\pi}{2}^-} \frac{1}{\(\cos t\)^{d/2}}  e^{-i
  \frac{|x|^2}{2}\tan t} e^{i\left\lvert \frac{x}{\cos
  t}\right\rvert^2 \frac{1}{2\tan t}} \frac{1}{(i\tan t)^{d/2}}
  \widehat u_+ \( \frac{x}{\tan t \cos t}\). 
\end{equation*}
The last quantity is equal to
\begin{equation*}
  \frac{e^{-id\pi/4}}{(\sin t)^{d/2}} \widehat u_+ \( \frac{x}{\sin
  t}\)e^{i\frac{|x|^2}{2}\( \frac{1}{\cos t \sin t}-\frac{\sin t}{\cos t}\)},
\end{equation*}
which converges to $e^{-id\pi/4} \widehat u_+(\cdot)$ in $L^2(\R^d)$ as $t\to
\pi/2$. 

The case $t\to -\pi/2$ is similar, up to a symmetry with
respect to the origin, since $\sin (-\pi/2)=-1$. 
\end{proof}
Denote $u_-=W_-^{-1}\phi$. We have therefore
\begin{equation*}
v\(-\frac{\pi}{2},x\) = e^{+id\pi/4} {\mathcal F}\(
  u_-\)(-x)  \quad ;\quad v\(\frac{\pi}{2},x\)= e^{-id\pi/4}{\mathcal F}\(Su_-\)(x). 
\end{equation*}
\begin{concl}
If the solution $v$ to
\eqref{eq:nlsharmo} with $v_{\mid t=0}=\phi$ satisfies
\begin{equation}\label{eq:condper}
  v\(\frac{\pi}{2},x\)=e^{-id\pi/2+i\theta}v\(-\frac{\pi}{2},-x\),
\end{equation}
then 
$u_-=W_-^{-1}\phi$ verifies $S(u_-)=e^{i\theta}u_-$. 
\end{concl}

\subsection{More rotating points?}
\label{sec:more}

In the case $\sigma=2/d$, the nonlinearity in \eqref{eq:nlsharmo} is
autonomous: we may apply the lens transform back and forth, and change
the time origin. The following result is then straightforward:
\begin{proposition}\label{prop:more}
  Let $d\ge 1$ and $\sigma=2/d$. Let $v\in C(\R;\Sigma)$ solve
  \eqref{eq:nlsharmo}, and such that
  \begin{equation*}
    v(t+\pi,x) = e^{-id\pi/2+i\theta}v(t,-x),\quad \forall (t,x)\in
    \R\times \R^d. 
  \end{equation*}
Then for all $t\in \R$, $v(t,\cdot)\in\Sigma$ is such that $u_-^t=
W_-^{-1} v(t)$ satisfies
\begin{equation*}
  S\(u_-^t\)=e^{i\theta}u_-^t.
\end{equation*}
\end{proposition}
\begin{remark}
  Due to the gauge invariance of \eqref{eq:nls}, $S$ is also gauge
  invariant -- see \eqref{eq:invariant} -- and the above result may be
  relevant only on time 
  intervals of length (at most) $\pi$. 
\end{remark}

\subsection{The focusing case}
\label{sec:foc1}
 Consider the equation:
\begin{equation}\label{eq:foc}
  i\d_t  v +\frac{1}{2}\Delta v = \frac{\lvert x\rvert^2}{2} v -
  \lvert v\rvert^{4/d}v .
\end{equation}
For initial data in $\Sigma$, the existence of a unique solution
locally in time is well-known (see e.g. \cite{CazCourant}). Rather
than the possibility of finite time blow-up, our interest is:
\begin{lemma}\label{lem:deffoc}
  Suppose that \eqref{eq:foc} has a solution $v\in
  C([-\pi/2,\pi/2];\Sigma)$. Then $u$, defined by
\begin{equation}\label{eq:inverse}
u(t,x)=\frac{1}{\(1+ t^2 \)^{d/4}}e^{i
\frac{
t}{1+t^2}\frac{|x|^2}{2} }v\( \operatorname{arctan} t,
\frac{x}{\sqrt{ 1 + t^2}}\)
\end{equation}
solves \eqref{eq:nlsfoc}. It satisfies $u\in C(\R;\Sigma)$, and has
asymptotic states in $\Sigma$, given by:
\begin{equation*}
  u_\pm(x) = e^{\pm id\pi/4}{\mathcal F}^{-1} v\(\pm \frac{\pi}{2}\)(\pm x). 
\end{equation*}
\end{lemma}
\begin{proof}
  We have immediately $u\in C(\R;\Sigma)$, and the fact that it solves
  \eqref{eq:nlsfoc}. To see that $u$ has asymptotic states, given by
  the above formula, write, for large $|t|$:
  \begin{align*}
    & U_0(-t)u(t,x)= \frac{1}{(-2i\pi t)^{d/2}}\int
    e^{-i\frac{|x-y|^2}{2t}} u(t,y)dy\\
&=\frac{1}{(-2i\pi t)^{d/2}}\int \frac{1}{\(1+ t^2 \)^{d/4}}
   e^{i
\frac{
t}{1+t^2}\frac{|y|^2}{2} } e^{-i\frac{|x-y|^2}{2t}} v\(
    \operatorname{arctan} t, \frac{y}{\sqrt{ 1 + t^2}}\)dy\\
&\approx  \frac{1}{(-2i\pi
    t)^{d/2}}\frac{1}{|t|^{d/2}}e^{-i\frac{|x|^2}{2t}} \int
    e^{i\frac{x\cdot y}{t}}e^{i\(-\frac{1}{t}+
    \frac{t}{1+t^2}\)\frac{|y|^2}{2}} v\(
    \operatorname{arctan} t, \frac{y}{\sqrt{ 1 + t^2}}\)dy\\
&\approx  \frac{e^{\pm
    id\pi/4}}{(2\pi)^{d/2}}\frac{1}{|t|^{d}}e^{-i\frac{|x|^2}{2t}} \int 
    e^{i\frac{x\cdot y}{t}}e^{-\frac{i}{t(1+t^2)}\frac{|y|^2}{2}} v\(
    \operatorname{arctan} t, \frac{y}{|t|}\)dy\\
&\approx  e^{\pm
    id\pi/4} e^{-i\frac{|x|^2}{2t}} {\mathcal F}^{-1} v\(
    \operatorname{arctan} t, \pm x\)\approx e^{\pm
    id\pi/4}  {\mathcal F}^{-1} v\(
    \pm\frac{\pi}{2}, \pm x\).
  \end{align*}
We have presented the computations in a formal way. We leave their
easy justification to the reader. 
\end{proof}

\section{Construction of periodic solutions}
\label{sec:periodic}

We construct a solution of the form 
$$v(t,x)=e^{-i\nu t}\psi(x).$$
In the defocusing case, $\psi$ must solve
\begin{equation}
  \label{eq:stat}
  \nu\psi = H\psi + |\psi|^{4/d}\psi.
\end{equation}
In the focusing case, it must solve
\begin{equation}
  \label{eq:statfoc}
 \( H-\nu\) \psi = \lvert \psi\rvert^{4/d}\psi.
\end{equation}
We have:
\begin{proposition}\label{prop:cri}
  Let $d\ge 1$.\\
$1.$  If $\nu>d/2$, then there exists an even function $\psi\in
  \Sigma\setminus\{0\}$ solving \eqref{eq:stat}.\\
$2.$ If $\nu<d/2$, then there exists an even function $\psi\in
  \Sigma\setminus\{0\}$ solving \eqref{eq:statfoc}.
\end{proposition}
Theorem~\ref{theo:main} follows from the first point, by considering
the family (for $\theta\in [0,2\pi[$)
\begin{equation*}
  (\nu_j)_{j\ge 1}=\Big\{\frac{d}{2}+2j-\frac{\theta}{\pi},\ j\in
  \N\setminus\{0\}\Big\}. 
\end{equation*}
The form of the corresponding
solution $u$ given in Remark~\ref{rem:forme} is straightforward, by
inverting the lens transform \eqref{eq:lens} (see \eqref{eq:inverse}).  
\smallbreak

To infer the existence part of Theorem~\ref{theo:foc}, we can consider
the family (for $\theta\in [0,2\pi[$)
\begin{equation*}
  (\nu_j)_{j\ge 1}=\left\{ \frac{d}{2}-2j-\frac{\theta}{\pi},\ j\in
  \N\setminus\{0\}\right\}. 
\end{equation*} 
\begin{remark}
  For such solutions, Proposition~\ref{prop:more} is
  irrelevant. Consider $e^{-i\nu_j  t}\psi(x)$ at two different
  times: this amounts to a multiplication by $e^{i\eta}$ for some
  $\eta\in \R$. As we have seen, the gauge invariance of
  \eqref{eq:nls} implies 
  that $S$ is also gauge invariant, and Proposition~\ref{prop:more}
  holds trivially for such solutions.
\end{remark}
\begin{remark}
  Proposition~\ref{prop:cri} remains valid if $|\psi|^{4/d}\psi$ is
  replaced with $|\psi|^{2\sigma}\psi$,
where the nonlinearity is $H^1$-subcritical, that
is, $\sigma<2/(d-2)$ when $d\ge 3$. However, since the nonlinearity in
\eqref{eq:nlsharmo} is autonomous if and only if $\sigma=2/d$, this is
the only case where it is reasonable to seek a solution to
\eqref{eq:nlsharmo} of the form $v(t,x)=e^{-i\nu t}\psi(x)$. 
\end{remark}
At least two proofs of Proposition~\ref{prop:cri} are available in the
literature: 
\begin{itemize}
\item In \cite{RW88}, the case of \eqref{eq:statfoc} is considered, by
  using bifurcation theory. As indicated there, the arguments
  presented in \cite{NirLN} make it possible to infer
  Proposition~\ref{prop:cri}. 
\item In \cite{KavianWeissler}, Proposition~\ref{prop:cri} is
  established up to the symmetry property (which could easily be
  incorporated): see \cite[Theorem~1.4]{KavianWeissler} for the first
  point, and \cite[Theorem~1.3]{KavianWeissler} for the second
  one. The proof there is based on the mountain pass lemma. 
\end{itemize}
Even though this result has been established elsewhere, we present a
another short, self-contained proof, for the sake of completeness. 
\begin{proof}[Proof of Proposition~\ref{prop:cri}]
   We proceed in the same spirit as in \cite{BL83a,CGM78}: let 
  \begin{align*}
    I(\psi)&= \frac{1}{2}\<H\psi,\psi\> -\frac{\nu}{2}\<\psi,\psi\>,\\
 M&=\left\{ \psi\in \Sigma,\ \psi(x)=\psi(|x|)\ ; \ 
    \frac{1}{1+2/d} \int_{\R^d}|\psi(x)|^{2+4/d}dx=1\right\}. 
  \end{align*}
We consider radially symmetric functions for simplicity; in
particular, these are even functions. 
The following lines essentially show that the negative part of $I$
can be controlled. 
Denote 
\begin{equation*}
  \delta =\inf_{\psi\in M}I(\psi).
\end{equation*}
{\bf First case: $\nu>d/2$.}\\
We show that  $0>\delta >-\infty$. To see that  $\delta<0$, 
consider $\psi(x)=ce^{-|x|^2/2}$, where $c$ 
is such that $\psi\in M$, and recall that $\psi$ is the unique
eigenfunction associated to $\l_0=d/2$: $H\psi=d/2\psi$.

Suppose that we could find 
  sequences in $M$ along which $I$ goes to
  $-\infty$. Let $(\psi_n)_{n\in \N}$ be such a sequence; necessarily,
  it is unbounded in $L^2(\R^d)$. We remark that $(x\psi_n)_{n\in \N}$
  and $(\nabla \psi_n)_{n\in \N}$ are also unbounded in $L^2(\R^d)$,
  with norms of the same order as $\|\psi_n\|_{L^2}$. To be more
  precise, we introduce a notation: let $(\alpha_n)_{n\in\N}$ and
  $(\beta_n)_{n\in\N}$ be two families of positive real numbers. 
\begin{itemize}
\item We write $\alpha_n \ll \beta_n$ if
$\displaystyle \limsup_{n\to +\infty}\alpha_n/\beta_n =0$.
\item We write $\alpha_n \lesssim \beta_n$ if 
$\displaystyle \limsup_{n\to +\infty}\alpha_n/\beta_n <\infty$.
\item We write $\alpha_n \approx \beta_n$ if $\alpha_n \lesssim
  \beta_n$ and $\beta_n \lesssim \alpha_n$. 
\end{itemize}
Let $\phi_n\in \Sigma$, with
 $\|\phi_n\|_{L^2}\to +\infty$. In general, up to extracting a
 subsequence, two possibilities can be distinguished: 
\begin{itemize}
\item $\|\nabla\phi_n\|_{L^2}\gg \|\phi_n\|_{L^2}$ and/or
  $\|x\phi_n\|_{L^2}\gg \|\phi_n\|_{L^2}$, 
\item Or $\|\nabla\phi_n\|_{L^2}\approx \|x\phi_n\|_{L^2}\approx
  \|\phi_n\|_{L^2}$. 
\end{itemize}
This stems from the uncertainty principle
\begin{equation*}
 \|\phi\|_{L^2}^2\le \frac{2}{d}\|\nabla\phi\|_{L^2} \|x\phi\|_{L^2}.
\end{equation*}
In our case, it is easy to see that the first possibility
leads to a contradiction, since $I(\psi_n)$ would be positive for $n$
sufficiently large. Consider the last possible case:
$\|\nabla\psi_n\|_{L^2}\approx 
  \|x\psi_n\|_{L^2}\approx \|\psi_n\|_{L^2}$. Introduce
\begin{equation*}
    \widetilde \psi_n = \frac{1}{\|\psi_n\|_{L^2}}\psi_n.
  \end{equation*}
This is a bounded sequence in $\Sigma$, whose $L^2$ norm is equal to
one. Up to extracting a subsequence, $\widetilde \psi_n$ converges weakly in
$\Sigma$. Since $\Sigma \hookrightarrow L^{2}(\R^d)$ is
compact, (a subsequence of) $\widetilde \psi_n$ converges strongly in
$L^2(\R^d)$, to some $\widetilde \psi \in  
\Sigma$ such that $\|\widetilde \psi\|_{L^2}=1$. Since the embedding
$\Sigma \hookrightarrow L^{2+4/d}(\R^d)$ is compact, $\widetilde
\psi_n \to \widetilde \psi$ strongly in $L^{2+4/d}(\R^d)$. We infer
\begin{equation*}
  \|\psi_n \|_{L^{2+4/d}} = \|\psi_n
  \|_{L^{2}} \|\widetilde \psi_n \|_{L^{2+4/d}}\approx \|\psi_n
  \|_{L^{2}}\to +\infty. 
\end{equation*}
Therefore, $\psi_n$ cannot remain in $M$, hence the finiteness of $\delta$.

Since the embedding $\Sigma \hookrightarrow L^p(\R^d)$ is compact
for $2\le p<2d/(d-2)$ ($2\le p\le \infty$ if $d=1$ and $2\le p<\infty$
if $d=2$), we infer that this infimum
is actually a minimum, attained by a non-trivial function $\psi\in
\Sigma$. Indeed, from what we have seen above, any minimizing sequence
is bounded in $L^2(\R^d)$, and therefore in $\Sigma$ since
$\delta<0$. The Lagrange multiplier $\mu$ associated to this
problem is such that 
\begin{equation*}
  H\psi -\nu \psi = \mu |\psi|^{4/d}\psi. 
\end{equation*}
The scalar product with $\psi$ yields $\mu <0$, since $\delta<0$:
\begin{equation*}
  \nu \psi =H\psi + |\mu| |\psi|^{4/d}\psi. 
\end{equation*}
The function $|\mu|^{d/4}\psi (\not =0)$ solves \eqref{eq:stat}. 
\smallbreak

\noindent {\bf Second case: $\nu<d/2$.}\\
 We show that $0<\delta<\infty$. The finiteness of $\delta$ is obvious,
and we recall that the uncertainty principle yields
\begin{equation}\label{eq:17h07}
  I(\psi)\ge \frac{1}{2}\(\frac{d}{2}-\nu\)\|\psi\|_{L^2}^2>0. 
\end{equation}
Assume that $\delta=0$: we can find a minimizing sequence $\psi_n\in M$
such that $\psi_n\to 0$ in $L^2$. Therefore, 
\begin{equation*}
  0\leftarrow I(\psi_n)= \frac{1}{2}\<H\psi_n,\psi_n\>+o(1),
\end{equation*}
and $\psi_n\to 0$ in $\Sigma$. This implies $\psi_n\to 0$ in
$L^{2+4/d}$: this contradicts $\psi_n\in M$, and so, $\delta>0$. We
see from \eqref{eq:17h07} that any minimizing
sequence is bounded in $L^2$, and thereby in $\Sigma$. Up to a
subsequence, such a sequence 
converges weakly in $\Sigma$, and strongly in $L^2\cap L^{2+4/d}$, to
a function $\psi\in M$ which verifies
\begin{equation*}
  \(H-\nu\)\psi=\mu |\psi|^{4/d}\psi.
\end{equation*}
Since $\delta>0$, taking the scalar product with $\psi$ in the above
equation shows that the Lagrange multiplier $\mu$ is positive. The
function $\mu^{d/4}\psi$ then solves \eqref{eq:statfoc}, hence the
proposition. 
\end{proof}

\section{End of the proof of Theorem~\ref{theo:foc}}
\label{sec:focusing}

Recall that for the existence part of  Theorem~\ref{theo:foc}, we
apply the second point of Proposition~\ref{prop:cri} with the family
(for $\theta\in [0,2\pi[$)
\begin{equation*}
  (\nu_j)_{j\ge 1}=\left\{ \frac{d}{2}-2j-\frac{\theta}{\pi},\ j\in
  \N\setminus\{0\}\right\}. 
\end{equation*}
Denote by $(\phi_j)_{j\ge 1}$ a corresponding family of even, nontrivial 
solutions to
\eqref{eq:statfoc}. The fact that this family is unbounded in
$H^1(\R^d)$ follows by taking the scalar product of \eqref{eq:statfoc}
with $\phi_j$, and invoking the Gagliardo--Nirenberg inequality
\eqref{eq:GNopt}: 
\begin{equation*}
  \<\(H-\nu_j\)\phi_j,\phi_j\>= \int_{\R^d}|\phi_j(x)|^{2+4/d}dx\le
  \frac{d+2}{2d\|Q\|_{L^2}^{4/d}}  
  \|\phi_j\|_{L^2}^{4/d}\|\nabla \phi_j\|_{L^2}^{2}. 
\end{equation*}
This reads:
\begin{equation*}
   \frac{d+2}{2d\|Q\|_{L^2}^{4/d}}  
  \|\phi_j\|_{L^2}^{4/d}\|\nabla \phi_j\|_{L^2}^{2}\ge \frac{1}{2}\|\nabla
  \phi_j\|_{L^2}^2 + 
  \frac{1}{2}\|x\phi_j\|_{L^2}^2 +
  \(2j-\frac{d}{2}+\frac{\theta}{\pi}\)\|\phi_j\|_{L^2}^2. 
\end{equation*}
We have directly, for $2j>d/2-\theta/\pi$, 
\begin{equation*}
  \|\phi_j\|_{L^2}^{4/d} >\frac{d}{d+2}\|Q\|_{L^2}^{4/d}.
\end{equation*}
Therefore,
\begin{equation*}
  \|\phi_j\|_{L^2}^{4/d}\|\nabla \phi_j\|_{L^2}^{2} \ge
  \frac{2d}{d+2}\|Q\|_{L^2}^{4/d} 
  \(2j-\frac{d}{2}+\frac{\theta}{\pi}\)\|\phi_j\|_{L^2}^2 \ge C
  \(2j-\frac{d}{2}+\frac{\theta}{\pi}\). 
\end{equation*}
The last inequality shows that $(\phi_j)_{j\ge 1}$ is unbounded in
$H^1(\R^d)$. The first inequality shows that if $d\ge 2$, then
\begin{equation*}
  \|\nabla \phi_j\|_{L^2}^{2} \gtrsim j. 
\end{equation*}
In view of Lemma~\ref{lem:deffoc}, the
following result completes the proof of Theorem~\ref{theo:foc}.
\begin{proposition}
  Let $\nu<d/2$, and $v(t,x) =e^{-i\nu t}\psi(x)$, where $\psi$ solves
  \eqref{eq:statfoc}. There exists $\varepsilon>0$ such that if
  $\phi\in\Sigma$ satisfies
  $\|v(-\pi/2,\cdot)-\phi\|_{L^2}<\varepsilon$, then the solution
  $\widetilde v$ to the initial value problem
  \begin{equation}
    \label{eq:vtilde}
    i\d_t \widetilde v +\frac{1}{2}\Delta \widetilde v = \frac{\lvert
      x\rvert^2}{2} \widetilde v - \lvert \widetilde v\rvert^{4/d}
    \widetilde v\quad ;\quad \widetilde v_{\mid t=-\pi/2} = \phi
  \end{equation}
is such that $\widetilde v\in C([-\pi/2,\pi/2];\Sigma)$. 
\end{proposition}
\begin{proof}
  For $\phi\in \Sigma$, the local existence of a solution $\widetilde
  v$ in $\Sigma$ is standard; see e.g. \cite{CaAHP,CazCourant}. We prove
  that if $\varepsilon$ is sufficiently small, then this solution cannot
  blow-up on the time interval $[-\pi/2,\pi/2]$. The analysis in
  \cite{CaAHP} yields, since $v\in C([-\pi/2,\pi/2];\Sigma)$:
  \begin{equation*}
    v,xv,\nabla v \in L^p\([-\pi/2,\pi/2];L^q(\R^d)\),\quad \forall
    (p,q)\text{ admissible.}
  \end{equation*}
Consider the function $w=v-\widetilde v$. It solves
\begin{equation}\label{eq:r}
  i\d_t w = Hw + g(v+w)-g(v)\quad ;\quad w_{\mid
    t=-\pi/2}=v-\widetilde v_{\mid t=-\pi/2}, 
\end{equation}
where we have denoted $g(z)=|z|^{4/d}z$. To prove the proposition, it
suffices to show that $w \in C([-\pi/2,\pi/2];\Sigma)$. Let
$t\in[-\pi/2,\pi/2]$, and denote
$D_t=[-\pi/2,t]\times\R^d$. Strichartz estimates with the admissible
pair 
$(2+4/d,2+4/d)$ yield, along with H\"older inequality:
\begin{align*}
  \|w\|_{L^{2+4/d}(D_t)} &\le C \|v-\widetilde v_{\mid
    t=-\pi/2}\|_{L^2} \\
&\ + C
  \(\|w\|_{L^{2+4/d}(D_t)}^{4/d}
  +\|v\|_{L^{2+4/d}(D_t)}^{4/d}\)
  \|w\|_{L^{2+4/d}(D_t)}.  
\end{align*}
Note that the constant $C$ can be chosen independent of
$t\in[-\pi/2,\pi/2]$, by considering $I=[-\pi/2,\pi/2]$ in
Lemma~\ref{lem:strichartz}. By splitting $I$ into a finite number of
intervals $I_j$ such that 
\begin{equation*}
  C\|v\|_{L^{2+4/d}(I_j\times\R^d)}^{4/d}\le \frac{1}{2},
\end{equation*}
and repeating the same arguments finitely many times, 
we see that there exists $C_0$ such that
\begin{equation*}
  \|w\|_{L^{2+4/d}(D_t)} \le C_0 \|v-\widetilde v_{\mid
    t=-\pi/2}\|_{L^2} + C_0
  \|w\|_{L^{2+4/d}(D_t)}^{4/d+1}, 
\end{equation*}
for all $t\in I$ (this is essentially Gronwall lemma on a finite time
interval). Therefore, choosing $\|v-\widetilde v_{\mid 
    t=-\pi/2}\|_{L^2}$ sufficiently small, a bootstrap argument shows
  that $w\in L^{2+4/d}(I\times \R^d)$. 
\smallbreak

Since the operators $x$ and $\nabla$ do not commute with $U_H$, we may
introduce the operators
\begin{equation*}
  J(t) = x\sin t
-i \cos t \nabla\quad ;\quad  
K(t)=  x\cos  t+i  \sin  t \nabla.
\end{equation*}
These operators commute with $U_H$, act on gauge invariant
nonlinearities like derivatives, and satisfy the pointwise property
\begin{equation}\label{eq:15h40}
  |J(t)f|^2 + |K(t)f|^2 = |xf|^2 + |\nabla f|^2. 
\end{equation}
We refer to \cite{CaAHP} for more details. 
Applying the operators $J$ and $K$ to \eqref{eq:r}, Strichartz
and H\"older inequalities yield
 \begin{align*}
  &\|J w\|_{L^{2+4/d}(D_t)}+ \|K w\|_{L^{2+4/d}(D_t)} \le C
  \|v-\widetilde v_{\mid 
    t=-\pi/2}\|_{\Sigma} \\
 + C&
  \(\|w\|_{L^{2+4/d}(D_t)}^{4/d}
  +\|v\|_{L^{2+4/d}(D_t)}^{4/d}\)
  \(\|J v\|_{L^{2+4/d}(D_t)} + \|K v\|_{L^{2+4/d}(D_t)}\)  \\
 + C&
  \(\|r\|_{L^{2+4/d}(D_t)}^{4/d}
  +\|v\|_{L^{2+4/d}(D_t)}^{4/d}\)
  \(\|J w\|_{L^{2+4/d}(D_t)} + \|K w\|_{L^{2+4/d}(D_t)}\).  
\end{align*}
Splitting $I$ into intervals where
\begin{equation*}
  C\(\|w\|_{L^{2+4/d}(I_j\times\R^d)}^{4/d}
  +\|v\|_{L^{2+4/d}(I_j\times\R^d)}^{4/d}\)\le \frac{1}{2},
\end{equation*}
we infer that $Jw,K w \in L^{2+4/d}(I\times
\R^d)$. Applying Strichartz inequality with now $(p_1,q_1)=(\infty,2)$
and $(p_2,q_2)=(2+4/d,2+4/d)$, we see that $w,Jw,K w \in
L^\infty(I;L^2(\R^d))$, hence $w,xw,\nabla w \in
L^\infty(I;L^2(\R^d))$ from \eqref{eq:15h40}. The results in
\cite{CaAHP} imply that $w\in C([-\pi/2,\pi/2];\Sigma)$.  
\end{proof}

\appendix

\section{Profile decomposition and nonlinear superposition}
\label{sec:profile}

Consider first the energy-critical nonlinear Schr\"odinger equation in
space dimension $d=3$:
\begin{equation}\label{eq:nlsH1}
  i\d_t u +\frac{1}{2}\Delta u = \lvert u\rvert^{4}u\quad ;\quad x\in
  \R^3. 
\end{equation}
Before stating the results we want to recall from \cite{Keraani01},
introduce a definition:
\begin{definition}
  If $(h^\varepsilon_j,t^\varepsilon_j,x^\varepsilon_j)_{j \in \N}$ is a  family of
sequences in $\R_+\setminus\{0\}\times\R\times \R^3$, then
we say that 
$(h^\varepsilon_j,t^\varepsilon_j,x^\varepsilon_j)_{j 
\in \N}$  is an  orthogonal family if
\begin{equation*}
\limsup_{\varepsilon \to
0}\left(\frac{h^\varepsilon_j}{h^\varepsilon_k} +\frac{h^\varepsilon_k}{h^\varepsilon_j} + \frac{
  |t^\varepsilon_j - 
t^\varepsilon_k|}{(h^\varepsilon_j)^2}
+ \left| \frac{ x^\varepsilon_j -x^\varepsilon_k}{h^\varepsilon_j}\right|
\right) = \infty\, , \quad \forall j \neq k. 
\end{equation*}
\end{definition}
The main two results in \cite{Keraani01}, which we recall below, are
the Schr\"odinger analogues to the results in \cite{BG99} for the wave
equation \eqref{eq:nlw}. 
\begin{theorem}[Theorem 1.6 in \cite{Keraani01}]\label{theo:lin}
  Let $(\phi^\varepsilon)_{0<\varepsilon\le 1}$ be a bounded family in $\dot
  H^1(\R^3)$. Let $u^\varepsilon_{\rm lin}=e^{i\frac{t}{2}\Delta}\phi^\varepsilon$.
Then, up to a subsequence (still denoted by $u^\varepsilon_{\rm lin}$), there exist a
family $(h_j^\varepsilon)_{j\ge 1}$ of positive numbers, a family
$(t_j^\varepsilon,x_j^\varepsilon)_{j\ge 1}$ of vectors in $\R\times \R^3$, and a
family $(V_j)_{j\ge 1}$ of solutions to 
\begin{equation*}
  i\d_t V +\frac{1}{2}\Delta V=0,
\end{equation*}
such that:
\begin{itemize}
\item $(h^\varepsilon_j,t^\varepsilon_j,x^\varepsilon_j)_{j 
\in \N}$  is an  orthogonal family.
\item For every $\ell \ge 1$, 
  \begin{equation*}
    u^\varepsilon_{\rm lin}(t,x) = \sum_{j=1}^\ell \frac{1}{\sqrt{h_j^\varepsilon}}
    V_j \(
    \frac{t-t_j^\varepsilon}{(h_j^\varepsilon)^2},\frac{x-x_j^\varepsilon}{h_j^\varepsilon}\) +
    w_\ell^\varepsilon(t,x), 
  \end{equation*}
with 
\begin{equation*}
  \limsup_{\varepsilon\to 0} \|w_\ell^\varepsilon\|_{L^q(\R;L^r(\R^3))}\Tend \ell
  \infty 0, 
\end{equation*}
for every pair $(q,r)$ with $6\le r<\infty$ and $2/q+3/r=1/2$. 
\end{itemize}
\end{theorem}
In \cite{Keraani01}, we find, since every $\dot H^1$ solution to
\eqref{eq:nlsH1} is global in time \cite{CKSTTAnnals}:
\begin{theorem}[From \cite{Keraani01} and
  \cite{CKSTTAnnals}]\label{theo:pnl} 
  Under the same assumptions as in Theorem~\ref{theo:lin}, consider
  the solutions to 
    \begin{equation*}
    i\d_t u^\varepsilon+\frac{1}{2}\Delta u^\varepsilon =\lvert u^\varepsilon\rvert^4
    u^\varepsilon \quad
    ;\quad u^\varepsilon(0,x)=\phi^\varepsilon(x), 
  \end{equation*}
associated with the subsequence of Theorem~\ref{theo:lin}. Then 
\begin{equation*}
  u^\varepsilon(t,x) = \sum_{j=1}^\ell \frac{1}{\sqrt{h_j^\varepsilon}}
    U_j \(
    \frac{t-t_j^\varepsilon}{(h_j^\varepsilon)^2},\frac{x-x_j^\varepsilon}{h_j^\varepsilon}\) +
    w_\ell^\varepsilon(t,x)+r_\ell^\varepsilon(t,x),
\end{equation*}
with
\begin{equation*}
  \limsup_{\varepsilon \to 0} \(\|\nabla r_\ell^\varepsilon\|_{L^\infty(\R; L^2(\R^3))}
  +\|r_\ell^\varepsilon\|_{L^{10}(\R^4)} + \|\nabla
  r_\ell^\varepsilon\|_{L^{10/3}(\R^4)} \)\Tend \ell \infty 0, 
\end{equation*}
where $h_j^\varepsilon,t_j^\varepsilon,x_j^\varepsilon,w_\ell^\varepsilon$ are as in
Theorem~\ref{theo:lin}, and the nonlinear profiles $U_j$ are given by:
\begin{equation*}
  i\d_t U_j+\frac{1}{2}\Delta U_j =\lvert U_j\rvert^4
    U_j \quad
    ;\quad \left\lVert \nabla
    \(U_j-V_j\)\(-\frac{t_j^\varepsilon}{(h_j^\varepsilon)^2}\)\right\rVert_{L^2(\R^3)}
    \Tend \varepsilon 0 0. 
\end{equation*}
\end{theorem}
Note that according to the limit of $t_j^\varepsilon/(h_j^\varepsilon)^2$ in
$[-\infty,+\infty]$, the profile $U_j$ is defined either by a Cauchy
data, or by an asymptotic state. Roughly speaking, the contribution of
$w_\ell^\varepsilon$ is linear, since this function is the same as in the
linear profile decomposition of Theorem~\ref{theo:lin}, while
$r_\ell^\varepsilon$ is asymptotically small (thus linear) as $\ell \to
\infty$. All in 
all, (leading order) nonlinear effects are measured through the
nonlinear profiles $U_j$. The orthogonality property shows that the
interactions of the scaled profiles are negligible in the limit $\varepsilon
\to 0$. The large time behavior of $u^\varepsilon$ is given,
asymptotically as $\varepsilon\to 0$, by the superposition of the large time
behavior of the scaled nonlinear profiles. Since every profile $U_j$
possesses asymptotic states, we see that $S$ acts on each profile
separately (as $\varepsilon\to 0$).
\smallbreak

In the $L^2$-critical case \eqref{eq:nls}, the profile decomposition
at the $L^2$ level is not merely a recasting of its $\dot H^1$
counterpart, because Galilean invariance must be taken into account. A
profile decomposition was introduced in \cite{MerleVega98} in the case
$d=2$, then generalized to the case $d\le 2$ in \cite{CK07}, in such a
way that the improved Strichartz estimates in \cite{BegoutVargas} yield a
profile decomposition in $L^2(\R^d)$ associated to solutions of
\eqref{eq:nls} for all $d\ge 1$. Due the existence of an extra
invariance, we modify the notion of orthogonal scales and cores:
\begin{definition}
If $(h^\varepsilon_j,t^\varepsilon_j,x^\varepsilon_j,\xi^\varepsilon_j)_{j \in \N}$ is a  family of
sequences in $\R_+\setminus\{0\}\times\R\times \R^d\times \R^d$, then
we say that 
$(h^\varepsilon_j,t^\varepsilon_j,x^\varepsilon_j,\xi^\varepsilon_j)_{j 
\in \N}$  is an  orthogonal family if
\begin{equation*}
\limsup_{\varepsilon \rightarrow
0}\left(\frac{h^\varepsilon_j}{h^\varepsilon_k} +\frac{h^\varepsilon_k}{h^\varepsilon_j} + \frac{
  |t^\varepsilon_j - 
t^\varepsilon_k|}{(h^\varepsilon_j)^2}
+ \left| \frac{ x^\varepsilon_j -x^\varepsilon_k}{h^\varepsilon_j} +
\frac{t^\varepsilon_j\xi^\varepsilon_j -t^\varepsilon_k\xi^\varepsilon_k }{h^\varepsilon_j} 
 \right|
\right) = \infty\, , \quad \forall j \neq k. 
\end{equation*}
\end{definition}
\begin{theorem}[From
  \cite{MerleVega98,CK07,BegoutVargas}]\label{theo:profile} 
Let  $d\ge 1$ and $(\phi^\varepsilon)_{0<\varepsilon\le 1}$ be a bounded family in
$L^2(\R^d)$. Up to extracting a subsequence, we have:\\
\emph{i)} 
There exist an
orthogonal family $(h^\varepsilon_j,t^\varepsilon_j,x^\varepsilon_j,\xi^\varepsilon_j)_{j \in \N}$ 
in $ \R_+\setminus\{0\}\times\R\times \R^d\times \R^d$, and a family
$(\phi_j)_{j \in \N}$ bounded in~$L^2(\R^d) $, such that for every
$\ell \ge 1$, 
\begin{equation*}
\begin{aligned}
& e^{i\frac{t}{2}\Delta}\phi^\varepsilon =\sum_{j=1}^\ell P_j^\varepsilon (\phi_j)(t,x)
+r^\varepsilon_\ell(t,x)\, , \\
\text{where }&\ \ \ 
P_j^\varepsilon (\phi_j)(t,x) 
= e^{ix\cdot \xi_j^\varepsilon
-i\frac{t}{2}\lvert\xi^\varepsilon_j\rvert^2}  \frac{1}{(h^\varepsilon_j)^{d/2}} V_j \left(
\frac{t-t^\varepsilon_j}{(h^\varepsilon_j)^{2}}\virgp
\frac{x-x^\varepsilon_j-t\xi_j^\varepsilon}{h^\varepsilon_j}\right),\\
\text{with }&\ \ V_j(t)=e^{i\frac{t}{2}\Delta} \phi_j ,\quad
\text{and }\quad  
\limsup_{\varepsilon \to 0}\|r^\varepsilon_\ell\|_{L^{2+4/d}(\R\times\R^d)} \Tend \ell
{+\infty} 0 . 
\end{aligned}
\end{equation*}
Furthermore, for every $\ell \ge 1$, we have
\begin{equation}\label{eq:orth}
\Norm{\phi^\varepsilon}{L^2(\R^d)}^2 = \sum_{j=1}^\ell \Norm{\phi_j}{L^2(\R^d)}^2
+ \Norm{r^\varepsilon_\ell}{L^2(\R^d)}^2 +o(1) \quad \text{as }\varepsilon \to 0\, .
\end{equation}
\emph{ii)} If in addition the family $(\phi^\varepsilon)_{0<\varepsilon\le 1}$ is
bounded in $H^1(\R^d)$, or more generally if 
\begin{equation}\label{eq:echelle1}
\limsup_{\varepsilon\to 0}\int_{|\xi|>R} \left| \widehat \phi^\varepsilon(\xi)\right|^2
d\xi \to 0 \quad \text{as }R\to +\infty\, ,
\end{equation}
then for every $j\ge 1$, $h_j^\varepsilon \ge 1$, and $(\xi_j^\varepsilon)_\varepsilon$ is
bounded, $ |\xi_j^\varepsilon|\le C_j$.
\end{theorem}
Contrary to the case of \eqref{eq:nlsH1}, the global existence of
solutions to \eqref{eq:nls} in the critical space ($L^2$) is not known
so far, hence a slightly intricate statement (as in \cite{Keraani01}
for the $\dot H^1$ case, written
at a time where the global existence for \eqref{eq:nlsH1} was not
known): 
\begin{theorem}[From
  \cite{MerleVega98,CK07,BegoutVargas}]
Under the same assumptions as in Theorem~\ref{theo:profile}, consider
the solutions to
\begin{equation*}
  i\d_t u^\varepsilon +\frac{1}{2}\Delta u^\varepsilon = \lvert
  u^\varepsilon\rvert^{4/d}u^\varepsilon\quad ;\quad u^\varepsilon(0,x)=\phi^\varepsilon(x),
\end{equation*}
associated with the subsequence of
Theorem~\ref{theo:profile}. Consider the solution $U_j$ to
\eqref{eq:nls} such that 
\begin{equation*}
\Norm{\(U_j-V_j\)\(\frac{-t^\varepsilon}{(h^\varepsilon)^2}\)}{L^2(\R^n)}\Tend
\varepsilon 0 0 . 
\end{equation*}
Let $I^\varepsilon\subset \R$ be a family of open intervals containing the 
origin. The following statements are equivalent:
\begin{itemize}
\item[(i)] For every $j\ge 1$, we have
\begin{equation*}
\limsup_{\varepsilon \to 0 }\Norm{U_j}{L^{2+4/d}(I^\varepsilon_j\times
  \R^d)}<+\infty ,\quad 
\text{where }I_j^\varepsilon := (h_j^\varepsilon)^{-2}\(I^\varepsilon-t_j^\varepsilon\) . 
\end{equation*}
\item[(ii)] $\dis \limsup_{\varepsilon \to 0 }\Norm{u^\varepsilon}{L^{2+4/d}(I^\varepsilon\times
\R^d)}<+\infty $.
\end{itemize}
Moreover, if \emph{(i)} or \emph{(ii)} holds, then $u^\varepsilon = \dis
\sum_{j=1}^\ell N_j^\varepsilon(\phi_j) + r_\ell^\varepsilon + \rho_\ell^\varepsilon$, where
$r_\ell^\varepsilon$ is given by Theorem~\ref{theo:profile}, and:
\begin{align*}
\limsup_{\varepsilon \to 0} &\(\Norm{\rho_\ell^\varepsilon}{L^{2+4/d}(I^\varepsilon\times \R^d)} +
\Norm{\rho_\ell^\varepsilon}{L^\infty(I^\varepsilon;L^2( \R^d))}\) \Tend \ell {+\infty} 0
,\\
 N_j^\varepsilon(\phi_j) (t,x)&= e^{ix\cdot
   \xi_j^\varepsilon-i\frac{t}{2}\lvert\xi_j^\varepsilon\rvert^2} 
 \frac{1}{(h_j^\varepsilon)^{d/2}} U_j\( \frac{t-t_j^\varepsilon}{(h_j^\varepsilon)^2}\virgp
 \frac{x-x_j^\varepsilon-t\xi_j^\varepsilon}{h_j^\varepsilon}\) .
\end{align*}
\end{theorem}
If, as expected, one has $I_j^\varepsilon=I^\varepsilon=\R$ for all $j$, then this
result is the exact analogue of Theorem~\ref{theo:pnl}.

\end{document}